\documentclass[letterpaper,11pt]{amsart}
\usepackage{tikz-cd}

\usepackage{latexsym,array,delarray,amsthm,amssymb,epsfig}
\usepackage{caption}
\usepackage{subcaption}
\usepackage{amscd}

\theoremstyle{plain}
\newtheorem{thm}{Theorem}[section]
\newtheorem{lemma}[thm]{Lemma}

\newtheorem{cor}[thm]{Corollary}
\newtheorem{conj}[thm]{Conjecture}
\newtheorem*{thm*}{Theorem}
\newtheorem*{lemma*}{Lemma}
\newtheorem*{prop*}{Proposition}
\newtheorem*{cor*}{Corollary}
\newtheorem*{conj*}{Conjecture}

\theoremstyle{definition}
\newtheorem{defn}[thm]{Definition}
\newtheorem{ex}[thm]{Example}

\theoremstyle{remark}

\newcommand{\rr}{\mathbb{R}}
\newcommand{\cc}{\mathbb{C}}

\newcommand{\ind}{\mbox{$\perp \kern-5.5pt \perp$}}

\begin{document}

\title{Initial Ideals of Pfaffian Ideals}
\author{Colby Long}

\begin{abstract}

We resolve a conjecture about a class of binomial initial ideals of $I_{2,n}$, the ideal of the Grassmannian, Gr$(2,\mathbb{C}^n$), which are associated to phylogenetic trees. For a weight vector $\omega$ in the tropical Grassmannian, $in_\omega(I_{2,n}) = J_\mathcal{T}$ is the ideal associated to the tree $\mathcal{T}$. The ideal generated by the $2r \times 2r$ subpfaffians
of a generic $n \times n$ skew-symmetric matrix is precisely $I_{2,n}^{\{r-1\}}$, the 
$(r-1)$-secant of $I_{2,n}$. We prove necessary and 
sufficient conditions on the topology of $\mathcal{T}$ in order for $in_\omega(I_{2,n})^{\{2\}} = J_\mathcal{T}^{\{2\}}$. We also give a new class
of prime initial ideals of the Pfaffian ideals.
\end{abstract}

\maketitle

%
%
%

%
%
%

\section{Introduction}
\label{Introduction}

The $r$-secant variety of a projective variety
 $X \subseteq \mathbb{P}^{m - 1}$ is 
$X^{\{r\}} := \overline{ \{x_1 + \ldots + x_r : x_i \in X \}},$
where the closure is taken in the Zariski topology.
Similarly, we define the 
 \emph{$r$-secant ideal} of a 
homogenous ideal, $ I(X)^{\{r\}} := I(X^{\{r\}})$. 
Secant varieties and ideals are classic objects of study in algebraic geometry. They are also of statistical interest as the sets of distributions associated to many statistical models exhibit the structure of secant varieties \cite{Garcia2005, Drton2008, Sullivant2008}. 
Another common operation on ideals is to take 
an \emph{initial ideal} with respect to some weight 
vector.
An initial ideal shares many properties of the original ideal but
is often more easily studied combinatorially.

In \cite{Sturmfels2006}, the authors explore the relationship between the secant ideal of an initial ideal and the initial ideal of a secant ideal. In particular, they explore
under what conditions these operations commute.
 In this paper, we investigate the relationship between
secant ideals of initial ideals and initial ideals of secant ideals for a class of ideals in bijection with binary leaf-labeled trees which we call the \emph{Pl\"ucker tree ideals.}

The Pl\"ucker tree ideals are so named
because they can be constructed as initial ideals of the Pl\"ucker ideal, $I_{2,n}$, which is the vanishing ideal of the Grassmannian, $Gr(2,\mathbb{C}^n)$, in the Pl\"ucker coordinates. The secant ideals of the 
Pl\"ucker tree ideals are then initial ideals of the well-known Pfaffian ideals. 
We let $J_\mathcal{T}$ denote the 
Pl\"ucker tree ideal
associated to $\mathcal{T}$. These ideals are discussed in \cite{Speyer2004} where the following theorem is proven.

\begin{thm}
\cite{Speyer2004}
 Let $\mathcal{T}$ be a binary phylogenetic $[n]$-tree. There exists a weight vector $\omega \in \mathbb{R}^{n \choose 2}$ and a sign vector $\tau \in \{\pm1\}^{n \choose 2}$ such that $J_\mathcal{T} = \tau \cdot in_\omega(I_{2,n})$, where the sign vector multiplies coordinate $p_{ij}$ by $\tau_{ij}$.
\end{thm}

They also appear in \cite{Sullivant2008} which 
discusses how these ideals and their secants are connected to Gaussian graphical models and concludes with the following conjecture. 

\begin{conj} 
\label{sethsconjecture}\cite[Conjecture 7.10]{Sullivant2008} Let $\mathcal{T}$ be a binary phylogenetic $[n]$-tree , $\omega \in \mathbb{R}^{n \choose 2}$ a weight vector, and 
$\tau \in \{\pm1\}^{n \choose 2}$ a sign vector such that $J_\mathcal{T} = \tau \cdot in_\omega(I_{2,n})$, then $\tau \cdot in_\omega(I_{2,n}^{\{r\}}) = J_\mathcal{T}^{\{r\}}$.
\end{conj}

We show that this conjecture is not true for any $r$. In the case where $r=2$, we also prove the following theorem giving necessary and sufficient conditions on the topology of 
$\mathcal{T}$ for the conjecture to hold. 
In the course of doing so, we also furnish a new class of prime initial ideals of the Pfaffian ideals.

\begin{thm}  
\label{gtisecants} Let $\mathcal{T}$ be a binary phylogenetic $[n]$-tree, $\omega \in \mathbb{R}^{n \choose 2}$ a weight vector,
and $\tau \in \{\pm1\}^{n \choose 2}$ a sign vector such that 
$\tau \cdot in_\omega(I_{2,n}) = J_\mathcal{T}$, then 
$\tau \cdot in_\omega(I_{2,n}^{\{2\}}) = J_\mathcal{T}^{\{2\}}$  if and only if $\mathcal{T}$ has fewer than five cherries.
\end{thm}

That Conjecture 1.2 holds in any instance is perhaps
somewhat unexpected as it was shown in \cite{Sturmfels2006} that the operations of taking initial ideals and taking secant ideals do not in general commute even when the initial ideals are monomial. 
This has possible implications for phylogenetics as
there is a close similarity between  Conjecture \ref{sethsconjecture} and 
\cite[Conjecture 4.1.1]{Long2016}.  
The latter concerns initial ideals of secant ideals associated to binary leaf-labeled trees under the Cavender-Farris-Neyman (CFN) model. There is also a close relationship between the ideals involved as $J_\mathcal{T}$ can be viewed as the intersection of the ideal for the CFN model on the tree $\mathcal{T}$ with a coordinate subring.

The rest of this paper is devoted to proving Theorem \ref{gtisecants} and investigating possible extensions. Section \ref{GTI} establishes the necessary background and notation for
the Pl\"ucker tree ideals and the Pfaffian ideals. 
We conclude the section with a few results about initial ideals of Pfaffian ideals and outline the technique that we will use to prove Theorem 1.3. We show that to prove the theorem we need to show that a certain class of initial ideals of the Pfaffian ideals are prime and to construct lower bounds on the dimension of the secants of the Pl\"ucker tree ideals. 
Sections \ref{primeinitial} and \ref{pluckerdimensions} establish the primeness and dimension results respectively and enable us to give a short proof of the main theorem. Finally, in Section \ref{Beyond the Second Secant} we examine some of the possible extensions of Conjecture \ref{sethsconjecture} for higher order secant varieties.

\section{Pl\"ucker Tree Ideals}
\label{GTI}

A \emph{binary tree} is a connected acyclic graph in which every vertex is either degree one or three. We call a degree one vertex of a binary tree $\mathcal{T}$ a \emph{leaf}. 
If the leaves of $\mathcal{T}$ are labeled by a label set $X$ then $\mathcal{T}$ is a \emph{binary phylogenetic $X$-tree}. Most often in this paper we will consider binary phylogenetic $[n]$-trees where $[n] := \{1, \ldots, n\}.$ Our terminology and notation for trees will follow the conventions from phylogenetics found in \cite{Semple2003} and we refer the reader there for more details. 

For what follows it will be useful to have a standard planar embedding of our trees. 
If $\mathcal{T}$ is a binary phylogenetic $[n]$-tree
then inscribe a regular $n$-gon on the unit circle in 
$\mathbb{R}^2$ and choose a planar representation of 
$\mathcal{T}$ so that the leaves are located at the vertices of the $n$-gon. 
Label the leaves of $\mathcal{T}$ in increasing order clockwise around the circle. 
The induced 4-leaf subtrees of a tree are called quartets and a tree is uniquely determined by its quartets \cite{Semple2003}.
With a circular embedding of $\mathcal{T}$ as described, every induced quartet on the leaves $1 \leq i < j < k < l \leq n$ is either $ij|kl$ or $il|jk$. The notation $ij|kl$ indicates that the induced quartet is the 4-leaf tree with one non-leaf edge whose removal disconnects the leaves labeled by $i$ and $j$ from those labeled by $k$ and $l$.

For trees with such a circular embedding the vector $\tau$ in Conjecture \ref{sethsconjecture} and
Theorem \ref{gtisecants}  is equal to the all ones vector. 
Thus, for the rest of this chapter we will consider only trees embedded in this manner so that we can ignore the sign vector entirely. The tree pictured in Example \ref{tree labeling} is 
a binary phylogenetic $[15]$-tree .

 Let  $Z^n= \mathbb{C}[p_{ij} : 1 \leq i < j \leq n]$ and 
 $$I_{2,n} = \langle p_{ij}p_{kl} - p_{ik}p_{jl} + p_{il}p_{jk} : 1 \leq i < j < k < l \leq n \rangle\subseteq Z^n$$
  be the the ideal of quadratic Pl\"ucker relations. Let $\mathcal{T}$ be a binary phylogenetic $[n]$-tree and assign positive lengths to the edges of $\mathcal{T}$. The choice of edge lengths naturally induces a metric $d$ on the leaves of $\mathcal{T}$ where $d(i ,j)$ is the length of the unique path between $i$ and $j$. Let $\omega \in \mathbb{R}^{n \choose 2}$ be the vector with $\omega_{ij} = d(i,j)$ for $i <j$. 
Then the initial ideal with respect to this weight vector is 
$$in_\omega(I_{2,n}) = \langle p_{ik}p_{jl} - p_{il}p_{jk}: ij|kl \text{ is a quartet of } \mathcal{T} \rangle$$ \cite[Corollary 4.4]{Speyer2004}. We call $J_\mathcal{T} = in_\omega(I_{2,n})$ the \emph{Pl\"ucker tree ideal of $\mathcal{T}$}. Note that any choice of positive edge lengths for $\mathcal{T}$ yields the same initial ideal.

Corollary 4.4 from \cite{Speyer2004} also gives us a way to realize $J_\mathcal{T}$ as the kernel  of a homomorphism.
 Let 
 $\mathbb{C}[y] = \mathbb{C}[y_e : e \text{ is an edge of } \mathcal{T}]$ and 
 $\phi_\mathcal{T}: Z^n \rightarrow \mathbb{C}[y]$ 
 be the homomorphism that sends $p_{ij}$ to the product of all of the parameters $y_e$ corresponding to edges on the unique path from $i$ to $j$. Then $J_\mathcal{T}$ is the toric ideal $\text{ker}(\phi_\mathcal{T})$.



\subsection{Initial Ideals of Pfaffian Ideals}
\label{Initial Ideals of Pfaffian Ideals}

The determinant of a generic $2r \times 2r$ skew-symmetric matrix is the square of  a polynomial called the Pfaffian of the matrix.
Let $P(n,r)$ be the ideal generated by the $2r\times2r$ subpfaffians of a generic $n \times n$ skew-symmetric matrix $P = (p_{ij})$. 
Each $2r\times2r$ Pfaffian equation corresponds to a $2r$-element set $K \subseteq [n] := \{1, \ldots, n\}$. 
The terms appearing in each Pfaffian are then in bijection with
perfect matchings on the set $K$.
  The Pfaffian ideal $P(n,r)$ is the $(r-1)$-secant of the  Pl\"ucker ideal, that is $P(n,r) = I_{2,n}^{\{r-1\}}$.
This result, as well as background and examples for the Pfaffian ideals, can be found in \cite{Pachter2005}.
In this section, we will collect a number of facts about the Pfaffian ideals which will be useful for proving the results that follow.

\begin{defn} 
Let $p$ be the $2r \times 2r$ Pfaffian equation corresponding to perfect matchings on the set $\{i_1, \ldots, i_{2r}\}$ with $i_1< \ldots < i_{2r}$. The \emph{crossing monomial} of $p$ is the monomial $p_{i_1,i_{r+1}}p_{i_2,i_{r+2}}\ldots p_{i_r,i_{2r}}.$
\end{defn}

\begin{thm} \cite[Theorem 2.1]{Jonsson2007}
\label {pfaffiangb}
There exists a term order $\prec_{circ}$ on $Z^n$ that selects the crossing monomial as the lead term of the Pfaffian equations. Furthermore, the $2r \times 2r $ Pfaffians form a Gr\"obner basis for $I_{2,n}^{\{r - 1\}}$ with respect to this term order and   $$in_{\prec_{circ}}(I_{2,n}^{\{r -1 \}})=
  \langle p_{i_1,i_{r+1}}p_{i_2,i_{r+2}} \ldots p_{i_{r},i_{2r}} :  1 \leq i_1 < i_2 < \ldots < i_{2r} \leq n \rangle .$$ 
  
\end{thm}

We also have the following corollary.


\begin{cor} 
\label{gbforinitial}
Let $\mathcal{T}$ be a binary phylogenetic $[n]$-tree  and $\omega$ a term order for $Z^n$ derived from $\mathcal{T}$ as above. Then the initial forms of the $2r \times 2r$ Pfaffians with respect to $\omega$ form a Gr\"obner basis for $in_\omega(I_{2,n}^{\{r-1\}})$ with respect to $\prec_{circ}$ and hence generate $in_\omega(I_{2,n}^{\{r-1\}})$.
\end{cor}

\begin{proof} Since all of our trees are circularly embedded, for $1 \leq i < j < k < l \leq n,$  the $\omega$-weight of $p_{ik}p_{jl}$, $\omega(p_{ik}p_{jl})$, is greater than or equal to that of both $p_{ij}p_{kl}$ and $p_{il}p_{jk}$. 
Therefore, if we let $p$ be the $2r \times 2r$ Pfaffian equation with monomials corresponding to perfect matchings of the set
$\{i_1, \ldots, i_{2r}\}$ with 
$1 \leq i_1 < i_2 < \ldots < i_{2r} \leq n$, then $in_\omega(p)$ contains the term $p_{i_1,i_{r+2}}p_{i_2,i_{r+3}} \ldots p_{i_{r+1},i_{2r}}$. Thus, the term order $\prec_{circ}$ refines the weight vector $\omega$ \cite{Speyer2004}. The result follows from \cite[Corollary 1.9]{sturmfels1996grobner}. \end{proof}

\subsection{Outline of the Proof of Theorem \ref{gtisecants}}
\label{outline of proof}

It was established in \cite{Sturmfels2006} that 
for any term order, the initial ideal of a secant ideal is contained inside the secant of the initial ideal. Therefore, if $\mathcal{T}$ is a binary phylogenetic $[n]$-tree and $\omega$ is constructed 
from $\mathcal{T}$ as above we have the inclusion
\begin{equation*}
in_{\omega}(I_{2,n}^{\{2\}})\subseteq (in_{\omega}(I_{2,n}))^{\{2\}} = J_\mathcal{\mathcal{T}}^{\{2\}}.
\end{equation*}

If a prime ideal is contained in another ideal of the same dimension then the two ideals must be equal. Therefore, we can prove equality above if we can show that the ideal $in_{\omega}(I_{2,n}^{\{2\}})$ is prime and that $\dim(J_\mathcal{\mathcal{T}}^{\{2\}}) = \dim(in_{\omega}(I_{2,n}^{\{2\}}))$. Because of the containment, it will actually suffice to show that
$\dim(J_\mathcal{\mathcal{T}}^{\{2\}}) \geq 
\dim(in_{\omega}(I_{2,n}^{\{2\}}))$. This will be our approach for proving Theorem \ref{gtisecants}. In Section \ref{primeinitial}, we address the issue of primeness by using elimination theory and induction. In Section \ref{pluckerdimensions}, we use the tropical secant dimension approach of \cite{Draisma2008} to establish the dimension results.

\section{Prime Initial Ideals of the Pfaffian Ideals}
\label{primeinitial}

The first part of our proof of Theorem \ref{gtisecants} requires showing that for a weight vector constructed from a circularly embedded binary phylogenetic $[n]$-tree, 
$ in_{\omega}(I_{2,n}^{\{2\}})$ is prime. 
In fact, we obtain the much stronger result below giving an entire class of prime initial ideals for the Pfaffian ideals.

\begin{thm}
\label{primesecants}
 Let $\omega$ be a weight vector constructed from a circular embedding of a binary phylogenetic $[n]$-tree  $\mathcal{T}$. Then for all $r,n \in \mathbb{N}$, $in_\omega(I_{2,n}^{\{r\}})$ is a prime ideal.
\end{thm}

We will prove Theorem \ref{primesecants}
by utilizing induction and Lemma \ref{primelemma}.

\begin{lemma}
\label{primelemma}
 \cite[Proposition 23]{Garcia2005} Let $\mathbb{K}$ be a field and  $J \subseteq \mathbb{K}$ 
 be an ideal containing a polynomial $f = gx_1 + h$ with $g,h$ not involving $x_1$ and $g$ not a zero divisor modulo $J$.
  Let $J_1= J \cap \mathbb{K}[x_2, \ldots, x_n]$ be the elimination ideal. Then $J$ is prime if and only if $J_1$ is prime.
\end{lemma}

Before we begin, will first need to prove
the following two lemmas that ensure the existence of
polynomials in $in_\omega(I_{2,n}^{\{r\}})$
satisfying the conditions 
of Lemma \ref{primelemma}.


\begin{lemma}
\label{pjnappears}

Let $\mathcal{T}$ be a binary phylogenetic $[n + 1]$-tree  and $\omega$ a weight vector constructed from $\mathcal{T}$.
If $n \geq 2r + 1$, then for $2r < j \leq n$,
there exists a polynomial in $in_\omega(I_{2,n+1}^{\{r\}})$ in which
  $p_{j,n+1} $ occurs linearly.
\end{lemma}

\begin{proof}

Proving this lemma requires choosing a particular circular planar embedding of the tree $\mathcal{T}$ which we now describe.
\medskip

\noindent \underline{Case 1}: There exists a split $A|B$ in $\mathcal{T}$ such that $\#A = r+1$.

\medskip

Choose such a split and circularly label the leaves in $A$ clockwise by the labels $\{n+1, 1, 2, \ldots, r\}$ and then complete the circular labeling of $\mathcal{T}$. 
   Now consider the 
   $(2r + 2) \times (2r + 2)$ Pfaffian equation
 $p \in I_{2,n+1}^{\{r\}}$ that is the sum of monomials corresponding to perfect matchings on the set $\{1,2,\ldots, 2r, j, n+1\}$ with $2r < j \leq n$. 
 As in Corollary \ref{gbforinitial}, the monomial 
 $$p_{1,r+2}p_{2,r+3}\ldots p_{r-1,2r}p_{r,j}p_{r+1,n+1}$$ appears in $in_\omega(p)$. The restriction of $\mathcal{T}$ to the labels 
 $\{r,r+1,j,n+1\}$,
 $\mathcal{T}_{|\{r,r+1,j,n+1\}}$, is the 4-leaf tree with
 nontrivial split 
 $r(n+1)|(r+1)j$.
 Therefore, 
 $\omega(p_{r,j}p_{r+1,n+1}) = \omega(p_{r,r+1}p_{j,n+1})$. Therefore, 
 $in_\omega(p)$ also contains the monomial 
$$p_{1,r+2}p_{2,r+3}\ldots p_{r-1,2r}p_{r,r+1}p_{j,n+1}$$ of equal 
$\omega$ weight, and so 
$p_{j,n+1}$ occurs linearly in $I_{2,n+1}^{\{r\}}$.

\medskip

\noindent \underline{Case 2}: There does not exist a split $A|B$ in $\mathcal{T}$ such that $\#A = r+1$. 

\medskip

Since $\mathcal{T}$ is binary,
there exists a split $A|B$ such that $ r+1 < \#A \leq 2r$. 
Choose such a split with $\#A$ as small as possible. 
Consider $\mathcal{T}_{|A}$ as a rooted tree and starting on the side of the root with the greater number of leaves (if one exists), circularly label the leaves by $\{n+1, 1, 2, \ldots, r, r+1, \ldots ,\#A- 1\}$. Complete the labeling to a circular labeling of $\mathcal{T}$. Choosing either of the edges adjacent to the root in $\mathcal{T}_{|A}$ induces the split $(n+1)123\ldots k | (k+1) \ldots (\#A - 1)$ in $\mathcal{T}_{|A}$.
Notice also that $k<r$. Otherwise, either the set $\{(n+1),1,2,3,\ldots, k \}$ labels a split of $\mathcal{T}$ with exactly $r+1$ leaves, which we assumed was not true, or it labels a split with between $r+1$ and $2r$ leaves, contradicting that $A$ was chosen so that $\#A$ was as small as possible.

As before, for $2r<j \leq n$, consider the Pfaffian generator that is 
the sum of monomials corresponding to perfect matchings
on the set $\{1,\ldots, 2r, j, n+1\}$.
Then
$in_\omega(p)$ contains the monomial 
$\bold{m} = 
p_{1,r+2}p_{2,r+3}\ldots p_{r-1,2r}p_{r,j}p_{r+1,n+1}$. 
The monomial $p_{k, (k + r + 1)}p_{(r+1),(n+1)}$ divides 
$\bold{m}$ and 
we know that 
$\omega(p_{k, (r + 1)}p_{(k + r+1),(n+1)}) = 
 \omega(p_{k, (k + r + 1)}p_{(r+1),(n+1)})$ 
since removing the edge of 
$\mathcal{T}_{|A}$ adjacent to the 
root on the side labeled by leaves $\{(n+1),1,2,\ldots, k\}$
induces the quartet $(n+1)k|(r+1)(k+r+1)$.
Therefore, we can replace 
$p_{k, (k + r + 1)}p_{(r+1),(n+1)}$
in $\bold{m}$ by the equal weight term $p_{k, (r + 1)}p_{(k + r+1),(n+1)}$ to produce a
monomial $\bold{m}'$ of $in_\omega(p)$.

Notice that now $p_{r,j}p_{(k + r+1),(n+1)}|\bold{m}'$. 
Since $k \geq (\#A)/2 - 1$ and $r \geq (\#A)/2$,  
it must be that $k + r + 1 \geq \#A$. Therefore, 
the edge that splits $A|B$ in $\mathcal{T}$ also splits $r(n+1)|j(k+r+1)$, since the leaves in $A$ are labeled by 
$\{n+1, 1, 2, \ldots, r, r+1, \ldots ,\#A- 1\}$.
So we can replace
 $p_{r,j}p_{(k + r+1),(n+1)}$ in
$\bold{m}'$ with $p_{r,(k+r+1)}p_{j,(n+1)}$ to produce another monomial of $in_\omega(p)$. Thus, $p_{j,n+1}$ occurs linearly in $in_\omega(p)$.
\end{proof}

A particular labeling constructed from Case 2 of the previous lemma is demonstrated in the following example.

\begin{ex}
\label{tree labeling}

Let $\mathcal{T}$ be the tree pictured so that $(n + 1) = 15$ and let $r=4$. Choose $\omega$ to be any weight vector constructed from $\mathcal{T}$. 
To apply the construction in Case 1 of Lemma \ref{pjnappears}, we would need to find an edge in $\mathcal{T}$ that induces a split with exactly 5 leaves on one side. 
Since no such edge exists, we must find a split so that 
$6 \leq \#A \leq 8.$ Moreover, we must choose $A$ as small as possible. Removing the edge $e_1$ in this tree induces the correct split with $\#A = 7$.

Viewed as a rooted tree, $\mathcal{T}_{|A}$ has four leaves on one side of the root and three on the other. As per the lemma, we begin labeling on the side with $4$ leaves. In this example, $k=3$ and $(k +r + 1) = 8$.

\begin{figure}[h!]
  \label{pnjtree}
  \centering
    \includegraphics[width=7cm]{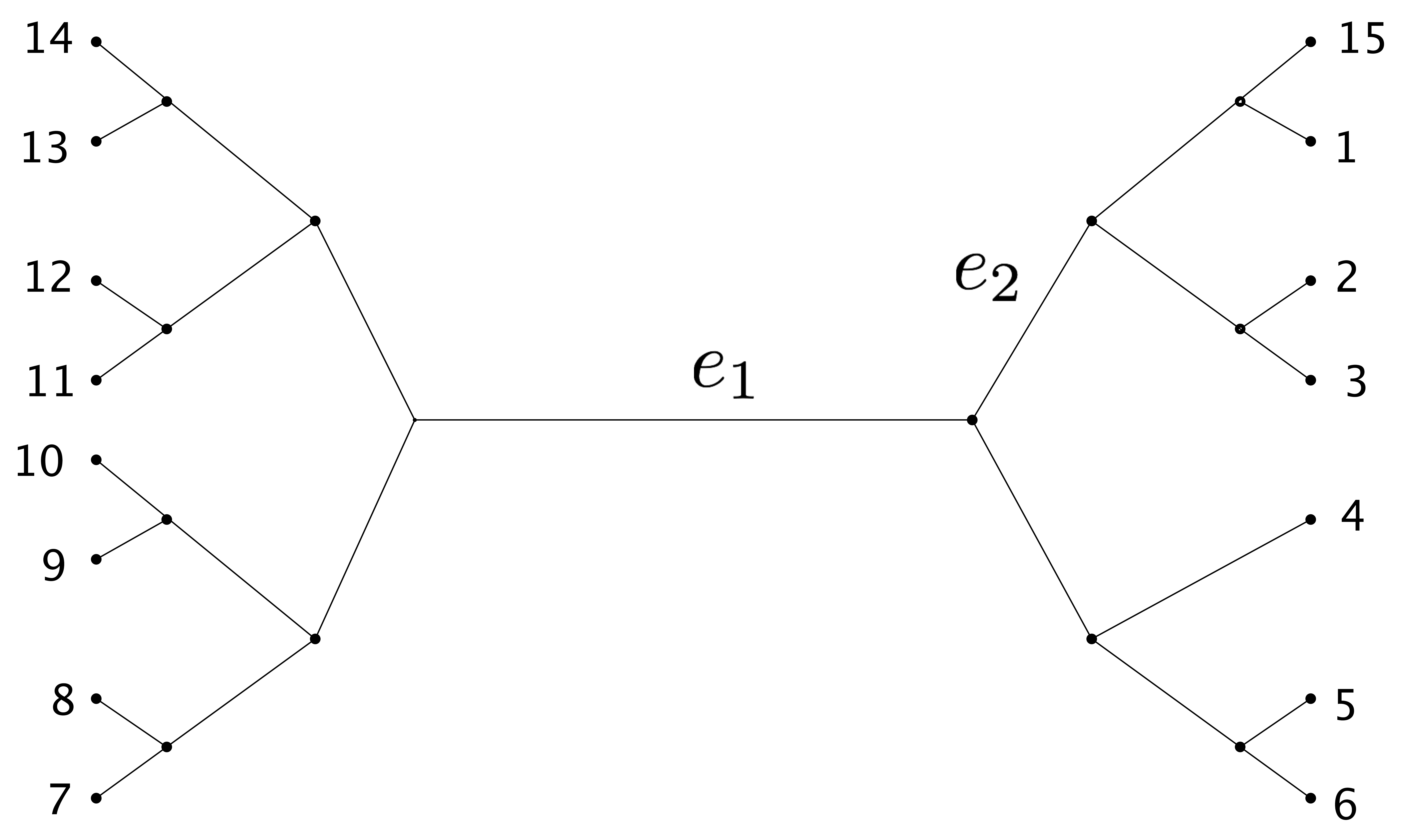}
\end{figure}


The variable $p_{j,n+1}$ occurs linearly in 
$in_\omega(I_{2,15}^{\{4\}})$ for any 
$8 <  j \leq 14$, but for the purposes of this example we will 
illustrate with $j = 14$. The crossing term of the 
Pfaffian equation on the set $\{1,2,3,4,5,6,7,8,14,15\}$ is 
$$p_{1,6}p_{2,7}p_{3,8}p_{4,14}p_{5,15}.$$ 

Removing the edge $e_2$ splits $\{3,15\}$ from $\{5,8\}$. Therefore, the Pfaffian equation includes the monomial of equal weight
$$p_{1,6}p_{2,7}p_{3,5}p_{4,14}p_{8,15}.$$ 
Finally, removing $e_1$ splits $\{8,14\}$ from $\{4,15\}$.
Thus, the monomial 
$$p_{1,6}p_{2,7}p_{3,5}p_{4,8}p_{14,15}$$ is also of equal
weight and so $p_{14,15}$ occurs linearly in this equation.

\end{ex}

\begin{lemma}
\label{zerodivisor}
If $n \geq 2r + 1$ then for $2r < j \leq n$ let $in_\omega(p)$ be the polynomial
found in Lemma \ref{pjnappears} in which $p_{j,n+1}$ occurs
linearly. Then $in_\omega(p) = g\cdot p_{j,n+1} + h$ with $g,h$ 
not involving $p_{j,n+1}$ and $g$ not a zero divisor modulo
$in_\omega(I_{2,n+1}^{\{r\}})$.
\end{lemma}

\begin{proof}

We write $in_\omega(p) = g\cdot p_{j,n+1} + h$ and observe that the polynomial $g$ is the sum of monomials corresponding to perfect matchings on the set $\{1,\ldots,2r\}$ with equal $\omega$-weight. In other words, $g \in in_{\omega'}(I_{2,n}^{\{r\}})$, where $\omega'$ is the subvector of $\omega$ without coordinates containing $(n+1)$ in the index.
So we just need to show that $g$ is not a zero divisor modulo
$in_\omega(I_{2,n+1}^{\{r\}})$.

Recall the term order $\prec_{circ}$ from Theorem \ref{pfaffiangb} with respect to which 
the Pfaffian equations form a Gr\"obner basis for 
$in_{\omega}(I_{2,n+1}^{\{r\}})$. 
Then 
$$in_{\prec_{circ}}(g) =   p_{1,r+1}p_{2,r+2}\ldots p_{r-2,2r-1}p_{r,2r}. $$

Suppose that there exists $g' \not \in in_\omega(I_{2,n+1}^{\{r\}})$ such that $gg' \in in_\omega(I_{2,n+1}^{\{r\}})$. Then choose such a $g'$ with standard leading term with respect to the Gr\"obner basis given by $\prec_{circ}$. Then 
$$in_{\prec{circ}}(gg') = (p_{1,r+1}p_{2,r+2}\ldots p_{r-2,2r-1}p_{r,2r}) in_{\prec{circ}}(g'),$$ 
and $in_{\prec{circ}}(gg')$ must be in 
$in_{\prec_{circ}}(I_{2,n+1}^{\{r\}})$.
Therefore, $in_{\prec{circ}}(gg')$ must be divisible by one of the 
crossing monomials which are the lead terms of the 
$(2r+2) \times (2r+2) $ Pfaffian equations.
But if $p_{ij}$ appears in the crossing monomial of a $(2r+2) \times (2r+2) $ Pfaffian equation, then $j-i \geq r+1$. This implies that 
$in_{\prec{circ}}(g)$ is relatively prime to every crossing monomial. Therefore,
$in_{\prec{circ}}(g')$ must be in the leading term ideal of 
$in_{\prec_{circ}}(I_{2,n+1}^{\{r\}})$ with respect to ${\prec_{circ}}$, which is a contradiction since we assumed it was standard.
\end{proof}

\begin{proof}[Proof of Theorem \ref{primesecants}]
We will proceed by induction.
 Fix $r \in \mathbb{N}$. 
 For $n < 2r + 1$, \\
$in_\omega(I_{2,n+1}^{\{r\}}) = \langle 0 \rangle$ which is prime. 

Now suppose $in_\omega(I_{2,n+1}^{\{r\}}) \subseteq 
 Z^{n+1}$ 
 is prime and consider the ideal
  $I_{2,n}^{\{r\}} \subseteq Z^{n}$. 
First, we show that
$(in_\omega(I_{2,n+1}^{\{r\}}) \cap Z^n) =
in_{\omega'}(I_{2,n}^{\{r\}})$, where again $\omega'$ is 
the subvector of $\omega$ that does not include coordinates
with $(n+1)$ in the index.
Define a grading on 
$Z^{n+1}$ where  $\deg(p_{ij}) = 1$ if $j=(n+1)$ and $\deg(p_{ij}) = 0$ otherwise.
Then $Z^{n+1} = \displaystyle \bigoplus_{i=0}^{\infty} Z^{n+1}_i$
and $I_{2,n+1}^{\{r\}}$ is homogeneous with respect to this
grading. 
It is true in general that for a homogeneous
ideal $I$ contained in a graded ring
$R = \displaystyle \bigoplus_{i=0}^{\infty} R_i$ and a weight 
vector $\omega$, that 
$I = \displaystyle \bigoplus_{i=0}^{\infty} I \cap R_i$ and
\begin{align*}
in_\omega(I) &= \displaystyle \bigoplus_{i=0}^{\infty} in_\omega(I \cap R_i ) \\
&= \displaystyle \bigoplus_{i=0}^{\infty} ( in_\omega(I) \cap R_i).
\end{align*}
In our case, we have 
$(in_\omega(I_{2,n+1}^{\{r\}}) \cap Z^{n+1}_0) =
in_\omega(I_{2,n+1}^{\{r\}} \cap Z^{n+1}_0)$. Since 
$(I_{2,n+1}^{\{r\}} \cap Z^{n+1}_0)
 = I_{2,n}^{\{r\}}$ and 
 $Z^n$ is precisely $Z^{n+1}_0$, the degree zero piece of $Z^{n+1}$, 
$(in_\omega(I_{2,n+1}^{\{r\}}) \cap Z^n) =
in_{\omega'}(I_{2,n}^{\{r\}})$.

So now assume the statement is true for all integers less than or equal to $n \geq 2r + 2$. 
We note by Lemma \ref{pjnappears} that each $p_{j,n+1}$ appears in some equation of
$in_\omega(I_{2,n+1}^{\{r\}})$.
Lemma \ref{zerodivisor} tells us that 
the coefficient of $p_{j,n+1}$ is not a zero divisor modulo $in_\omega(I_{2,n+1}^{\{r\}})$, but this also implies that 
each coefficient is not a zero divisor modulo any elimination
 ideal of $in_\omega(I_{2,n+1}^{\{r\}})$. 
So now beginning with $j = n$, we eliminate $p_{j,n+1}$ for $2r < j \leq n$ from 
$in_\omega(I_{2,n+1}^{\{r\}})$.
Importantly, 
the equation in which 
$p_{j,n+1}$ occurs linearly found in Lemma \ref{pjnappears}
does not contain any variables of the form $p_{k,n+1}$ for $k>j$ and so is still contained in the elimination ideal after we have eliminated all of these variables. 
Therefore, at each step, we meet the conditions of Lemma \ref{primelemma}, 
which implies that each successive elimination ideal is prime if and only if 
$in_\omega(I_{2,n+1}^{\{r\}})$ is prime.

 After eliminating, we have the ideal
 $in_\omega(I_{2,n+1}^{\{r\}}) 
 \cap Z^n[p_{1,n+1}, \ldots, p_{2r,n+1}]$ which we will 
 now show is equal to 
$in_\omega(I_{2,n+1}^{\{r\}}) \cap Z^n =
in_\omega(I_{2,n}^{\{r\}})$. In other words, we will show
that after eliminating $\{p_{2r + 1,n+1}, \ldots, p_{n,n+1}\}$,
there are no equations 
involving any variable with
$n+1$ in the index in the elimination ideal.
 Then by induction, the proof will be complete.

The dimension of $I_{2,n}^{\{r\}}$, and hence the dimension of all of its initial ideals, is $2rn - 2r^2 - r$ \cite{Kleepe1980}.
Since $I_{2,n+1}^{\{r\}}$ is prime, every irreducible component of
$in_\omega(I_{2,n+1}^{\{r\}})$ has dimension $2r(n+1) - 2r^2 - r$
\cite{Kalkbrener1995}.
The birational projection of Lemma \ref{primelemma} preserves
the dimension of each component, which implies
$$\dim(in_\omega(I_{2,n+1}^{\{r\}}) \cap 
 Z^n[p_{1,n+1}, \ldots, p_{2r,n+1}]) = 
 2r(n+1) - 2r^2 - r =
 \dim( in_\omega( I_{2,n}^{\{r\}})) + 2r.$$
 Therefore, eliminating the remaining $2r$ variables must 
 decrease the dimension of each component by
 $2r$, which implies
 that the variables in $\{p_{1,n+1}, \ldots, p_{2r,n+1}\}$ 
 are free in each component of 
 $in_\omega(I_{2,n+1}^{\{r\}}) \cap 
   Z^n[p_{1,n+1}, \ldots, p_{2r,n+1}]$. We conclude
 that  $in_\omega(I_{2,n+1}^{\{r\}}) \cap 
   Z^n[p_{1,n+1}, \ldots, p_{2r,n+1}]
   = in_{\omega'}(I_{2,n}^{\{r\}})$.
\end{proof}

\section{Dimensions of Secants of the Pl\"ucker Tree Ideals}
\label{pluckerdimensions}

To address Conjecture \ref{sethsconjecture} we first construct a simple bound on $\dim(J_\mathcal{T}^{\{r\}})$.

\begin{lemma}
\label{dimbound}
 Let $\mathcal{T}$ be a tree with $c$ cherries, then $\dim(J_\mathcal{T}^{\{r\}}) \leq 2rn - 3r - (r-1)c$.
\end{lemma}

\begin{proof}
%

 The variables corresponding to cherries do not appear in any of the binomials generating $J_\mathcal{T}$. 
 Thus, we can write 
 $V(J_\mathcal{T}) = V \times \mathbb{C}^c$ 
 and $V(J_\mathcal{T})^{\{r\}}= V^{\{r\}} \times \mathbb{C}^c$, 
 since $\mathbb{C}^c$ is a linear space. 
 The expected dimension of $V^{\{r\}}$ is $r\dim(V) + (r - 1)$. 
 However, $J_\mathcal{T}$ being homogeneous implies that 
 $V$ is a cone and that $\dim(V^{\{r\}}) \leq r\dim(V)$. Since they share the same Hilbert series, the dimension of $J_\mathcal{T}$ is equal to that of $I_{2,n}$ which is $2n - 3$ \cite{Kleepe1980}. Thus, we have 
\begin{align*}
\dim(V(J_\mathcal{T})^{\{r\}}) &\leq r\dim(V) + c \\
&= r(2n - 3 - c) + c. \\ 
&= 2rn - 3r - (r-1)c
\end{align*}
\end{proof}

\begin{cor} 
\label{cherrycorollary}
Conjecture \ref{sethsconjecture} does not hold for any $r$.
\end{cor}
\begin{proof}
Every initial ideal of $I_{2,n}^{\{r\}}$ has dimension $2rn - 2r^2 - r$ \cite{Kleepe1980}. Therefore, it is impossible for $J_\mathcal{T}^{\{r\}}=I_{2,n}^{\{r\}} $ if 
\begin{align*}
 2rn - 3r - (r-1)c &< 2rn - 2r^2 - r  \\
 -(r-1)c &< -2r^2 + 2r  \\
  c &>2r. \\
 \end{align*} 
 Thus, for any $r$, trees with more than $2r$ cherries serve as a counterexample. One can always construct such a tree 
 with $4r + 2$ leaves by simply attaching a cherry to each leaf in a tree with $2r + 1$ leaves.
 \end{proof}
 
The claim of Theorem \ref{gtisecants} is that when $r=2$, trees with strictly more than $4$ cherries are the only obstructions.
%
Before we begin the proof of Theorem \ref{gtisecants} we will discuss the specific structure of $I_{2,n}^{\{2\}}$ and the initial ideals $in_{\omega}(I_{2,n}^{\{2\}})$. The ideal $I_{2,n}^{\{2\}}$ is the vanishing ideal of the set of $n \times n$ rank four skew-symmetric matrices and is generated by the $6\times6$ Pfaffian equations. There are $n \choose 6$ of these degree 3  equations each with 15 terms corresponding to the perfect matchings on the 6-element subset of $[n]$ to which the equation corresponds.
Theorem \ref{pfaffiangb} tells us that the initial forms of these equations with respect to $\omega$ form a Gr\"obner basis for $in_\omega(I_{2,n}^{\{2\}})$. Without loss of generality, let $p$ be the $6\times6$ Pfaffian equation for the set 
$K = \{1,2,3,4,5,6\} \subseteq [n]$ and let $\mathcal{T}_{|K}$ be the restriction of $\mathcal{T}$ to the leaves of $K$. Up to relabeling of the leaves, there are only two 6-leaf tree topologies and the structure of $in_\omega(p)$ is completely determined by the topology of $\mathcal{T}_{|K}$.

 If $\mathcal{T}$ is the $6$-leaf caterpillar tree with nontrivial splits $12|3456$, $123|456$, and $1234|56$, then 
$$in_{\omega}(p) = p_{14}p_{25}p_{36}  - p_{14}p_{26}p_{35} - p_{15}p_{24}p_{36}   + p_{15}p_{26}p_{34} + p_{16}p_{24}p_{35}  - p_{16}p_{25}p_{34}. $$ 
If $\mathcal{T}$ is the $6$-leaf snowflake tree with nontrivial splits $12|3456$, $34|1256$, and $56|1234$, then 
\begin{align*}
in_{\omega}(p) = &p_{14}p_{25}p_{36}  - p_{14}p_{26}p_{35} - p_{15}p_{24}p_{36}   + p_{13}p_{25}p_{46} +  \\
&p_{16}p_{24}p_{35}  - p_{13}p_{26}p_{45} + p_{15}p_{23}p_{46}  - p_{16}p_{23}p_{45}.
\end{align*}

Thus, $in_\omega(I_{2,n}^{\{2\}})$ has a Gr\"obner basis 
consisting of $n \choose 6$ equations each with either six 
or eight terms. 
We call the binary phylogenetic $[n]$-tree with exactly two cherries the \emph{n-leaf caterpillar}. Although the following theorem for caterpillar trees does not generalize to a proof of Theorem \ref{gtisecants}, we include it because it is rather straightforward and establishes one of the base cases for our
inductive argument.

\begin{thm}  
\label{caterpillars} 
Let $\mathcal{C}$ be an $n$-leaf caterpillar tree and $\omega \in \mathbb{R}^{n \choose 2}$ be a weight vector such that $in_\omega(I_{2,n}) = J_\mathcal{C}$, then $in_\omega(I_{2,n}^{\{2\}}) = J_\mathcal{C}^{\{2\}}$.
\end{thm}

\begin{proof}  
Recall from above that the initial ideal of a secant ideal is contained inside the secant of the initial ideal \cite{Sturmfels2006}, so we have the inclusion,
\begin{equation}
in_{\omega}(I_{2,n}^{\{2\}})\subseteq (in_{\omega}(I_{2,n}))^{\{2\}} = J_\mathcal{C}^{\{2\}}.
\end{equation}
Let $P$ be the poset on the variables of $Z^n$ given by $p_{ij} \leq p_{kl}$ if $i \leq k$ and $j \leq l$ and $J(P)$ the monomial ideal generated by incomparable pairs $p_{ij}p_{kl}$ in $P$. 
There exists a term order $\omega'$ for which $in_{\omega'} (J_C)= J(P)$ \cite[Theorem 14.16]{Miller2005}. 
Taking initial ideals with respect to $\omega'$ in (1), we have
 \begin{equation}
in_{\omega'}(in_{\omega}(I_{2,n}^{\{2\}}))
\subseteq in_{\omega'}(J_\mathcal{C}^{\{2\}}) 
\subseteq (in_{\omega'}(J_\mathcal{C}))^{\{2\}}
= J(P)^{\{2\}}.
\end{equation}
%
%

In fact, there exists $\omega'' = \omega + \epsilon\omega'$ 
such that $ in_{\omega'}(in_{\omega}(I_{2,n}^{\{2\}})) =  
in_{\omega''}(I_{2,n}^{\{2\}})$ \cite[Proposition 1.13]{sturmfels1996grobner}.
It is also shown in \cite{Sturmfels2006} Example 4.13, 
that we can choose a term order $\prec$ for which 
$in_\prec(I_{2,n}^{\{2\}}) = J(P)^{\{2\}}$. 
This implies 
$$HS( J(P)^{\{2\}} ,t) = 
HS(in_\prec(I_{2,n}^{\{2\}}) ,t) = 
HS(in_{\omega''}(I_{2,n}^{\{2\}})) ,t),$$ which gives equality 
all across (2). This further implies that 
$$ HS(J_C^{\{2\}},t) = 
HS(in_{\omega''}(I_{2,n}^{\{2\}}),t) = 
HS(in_{\omega}(I_{2,n}^{\{2\}}),t)
,$$ giving equality in (1) and completing the proof. \end{proof}

\bigskip

Now to complete the proof of Theorem \ref{gtisecants}, 
we need only show that
for trees with exactly 3 or 4 cherries,
$\dim(in_{\omega}(I_{2,n}^{\{2\}})) = 
\dim( J_\mathcal{T}^{\{2\}}).$
Because of the containment
$\dim(in_{\omega}(I_{2,n}^{\{2\}})) \subseteq 
\dim( J_\mathcal{T}^{\{2\}}),$ 
this amounts to showing
 $ \dim( J_\mathcal{T}^{\{2\}}) \geq 2rn - 2r^2 - r$.
To do this, we will use the tropical secant dimension approach of \cite{Draisma2008}. We adapt the notation and terminology here for our purposes but refer the reader there for a complete description of the method.

Let $C_1, \ldots, C_r$  be affine cones. Suppose further that 
$C_i = \overline {\text{Im}(f_i)}$ where $f_i : \cc^{m_i} \rightarrow  \cc^{|B|}$ is a morphism. For $1 \leq i \leq r$, we write $f_i$ as a list $ (f_{ i , b })_{b\in B}$.
For our purposes in this paper, we may assume that each $f_{i,b}$ is a monomial, so that $f_{i , b} = x^{\alpha_{i,b}}.$  
For affine cones the mixing parameters introduced when constructing the join variety are superfluous. Thus, we can write the \emph{join} of the affine cones $C_1,\ldots, C_r$  
 as

$$
 C_1 + \ldots + C_r :=\overline{ \{ c_1 + \ldots + c_r : c_i \in C_i, 1 \leq i \leq r\}}. 
 $$

\begin{defn}
\label{winning}
For $ v = (v_1, \ldots, v_r )  \in \displaystyle \bigoplus_{i=1}^r \rr^{m_i} $, let
$$ 
D_i(v) := \{ \alpha_{i,b}  : \langle v_i, \alpha_{i,b} \rangle > \langle v_j, \alpha_{j,b} \rangle \text{ for all }  j \not = i  \}.
$$
If $\alpha_{i,b} \in D_i(v)$ then $\emph{i wins b at v} $ and we call $D_i(v)$ the set of \emph{winning directions of i at v}.
\end{defn}

\noindent Finally, the result below gives us a method of constructing lower bounds on the dimension of the join of affine cones.

\begin{lemma}
\label{lemma: Draisma}
\cite{Draisma2008}
The affine dimension of $ C_1 + \ldots + C_r $ is at least the maximum, taken over all 
$ v = (v_1, \ldots, v_r )  \in \displaystyle \bigoplus_{i=1}^r \rr^{m_i} $, of the sum 
$$\displaystyle\sum_{ i = 1 }^r {\rm dim}_{\rr} \langle D_i(v) \rangle_\rr .$$
\end{lemma}

Of course, we are actually interested in the 
dimension of the ideal $J_\mathcal{T}^{\{2\}}$.
 To apply the lemma, we regard the underlying
 projective variety $V(J_\mathcal{T})$ as an affine cone
 so that
  $\dim(J_\mathcal{T}^{\{2\}}) =
  \dim( V(J_\mathcal{T}) + V(J_\mathcal{T})).$ 
 So here, $r=2$, $C_1 = C_2 =  V(J_\mathcal{T})$,
 and $B = \{ \{k,l\} : 1 \leq k < l \leq n\}$.
 Recall from Section \ref{GTI} that $J_\mathcal{T}$ 
is the Zariski closure of the monomial map
 $\phi_\mathcal{T}: Z^n \rightarrow \mathbb{C}[y]$, 
 where $\phi_\mathcal{T}(p_{ij})$ is the square-free monomial parametrizing $p_{ij}$. 
Letting $\alpha_{ij}^\mathcal{T} \in \mathbb{R}^{2n - 3}$ be the $0/1$ coefficient vector of $\phi_\mathcal{T}(p_{ij})$
we have a simplified version of the lemma.

\begin{lemma}
\label{lemmadraismaJT}
The dimension of $ V(J_\mathcal{T}) + V(J_\mathcal{T}) $ is at least the maximum, taken over all $ v = (v_1, v_2)  \in  \rr^{2n - 3} \oplus \rr^{2n - 3} $, of the sum 
$$\dim_{\rr} \langle D^\mathcal{T}_1(v) \rangle_\rr  
 + \dim_{\rr} \langle D^\mathcal{T}_2(v) \rangle_\rr .$$
\end{lemma}

\begin{lemma}
\label{jtdraisma}
Let $\mathcal{T}$ be a binary phylogenetic $[n]$-tree with exactly 3 or 4 cherries, then 
$\dim(J_\mathcal{T}^{\{2\}}) \geq  4n - 10$.
\end{lemma}

\begin{proof} We will  prove by induction on $n$ that there exists a vector 
$v=(v_1, v_2) \in  \rr^{2n - 3} \oplus \rr^{2n - 3}$ 
such that 
$\dim_{\rr} \langle D^\mathcal{T}_1(v) 
\rangle_\rr = \dim_{\rr} \langle D^\mathcal{T}_2(v) \rangle_\rr  = 2n - 5$. 
First, note that every tree with exactly 3 cherries can be constructed by successively attaching leaves to the snowflake tree so that the new leaf is not involved in a cherry. Every tree with exactly 4 cherries can be constructed in the same manner from the unique 8-leaf tree with 4 cherries. By random search, we can find vectors that give us the lower bound for these two trees establishing our base cases. These vectors and the computations to verify the lower bounds can be found in the Maple worksheet ${\tt{SecantDimension.mw}}$ 
located at the author's website.

Assume the statement is true for all binary phylogenetic $[n]$-trees and let $\mathcal{T}$ be a binary phylogenetic $[n + 1]$-tree  with exactly 3 or 4 cherries. 
Label $\mathcal{T}$ so that the leaf labeled by $(n+1)$ is not part of a cherry. 
Let $\mathcal{R} = \mathcal{T}_{|[n]}$ be the tree obtained by
restricting $\mathcal{T}$ to the leaves labeled by $[n]$ and 
deleting the resulting degree two vertex.
By our inductive assumption, there exists 
 $v=(v_1, v_2) \in  \rr^{2n - 3} \oplus \rr^{2n - 3}$ 
 such that 
 $\dim_{\rr} \langle D^\mathcal{R}_1(v) \rangle_\rr 
 = \dim_{\rr} \langle D^\mathcal{R}_2(v) \rangle_\rr 
  = 2n - 5$. 
  Our goal will be to construct a new vector $w = (w_1, w_2) \in   \rr^{2(n+1) -3 } \oplus \rr^{2(n+1) -3 }$ so that
   $\dim_{\rr} \langle D^\mathcal{T}_1(w) \rangle_\rr 
   = \dim_{\rr} \langle D^\mathcal{T}_2(w) \rangle_\rr  = 2(n+1) - 5$. 
   
When adding the $(n+1)$ leaf to $\mathcal{R}$, we introduce ``new" edges $e_a$, $e_b$, and $e_{n+1}$ and eliminate the edge $e$. Let $u_a$ be the vertex of $e_a$ not shared with
 $e_b$ and likewise let  $u_b$ be the vertex of $e_b$ not shared with $e_a$.
Arbitrarily choose two leaves $L_1$ and $L_2$ such that the path from these leaves to $(n+1)$ passes through $e_a$. 
Also choose leaves $L_3$ and $L_4$ such that the path from these leaves to $(n+1)$ passes through $e_b$. 
Such leaves exist since $(n+1)$ is not contained in a cherry. 
Figure \ref{diminductiontree} depicts the situation.

\begin{figure}[h!]
  \caption{An example of the labeling scheme described in Lemma \ref{jtdraisma}.}
  \label{diminductiontree}
  \centering
    \includegraphics[width=0.4\textwidth]{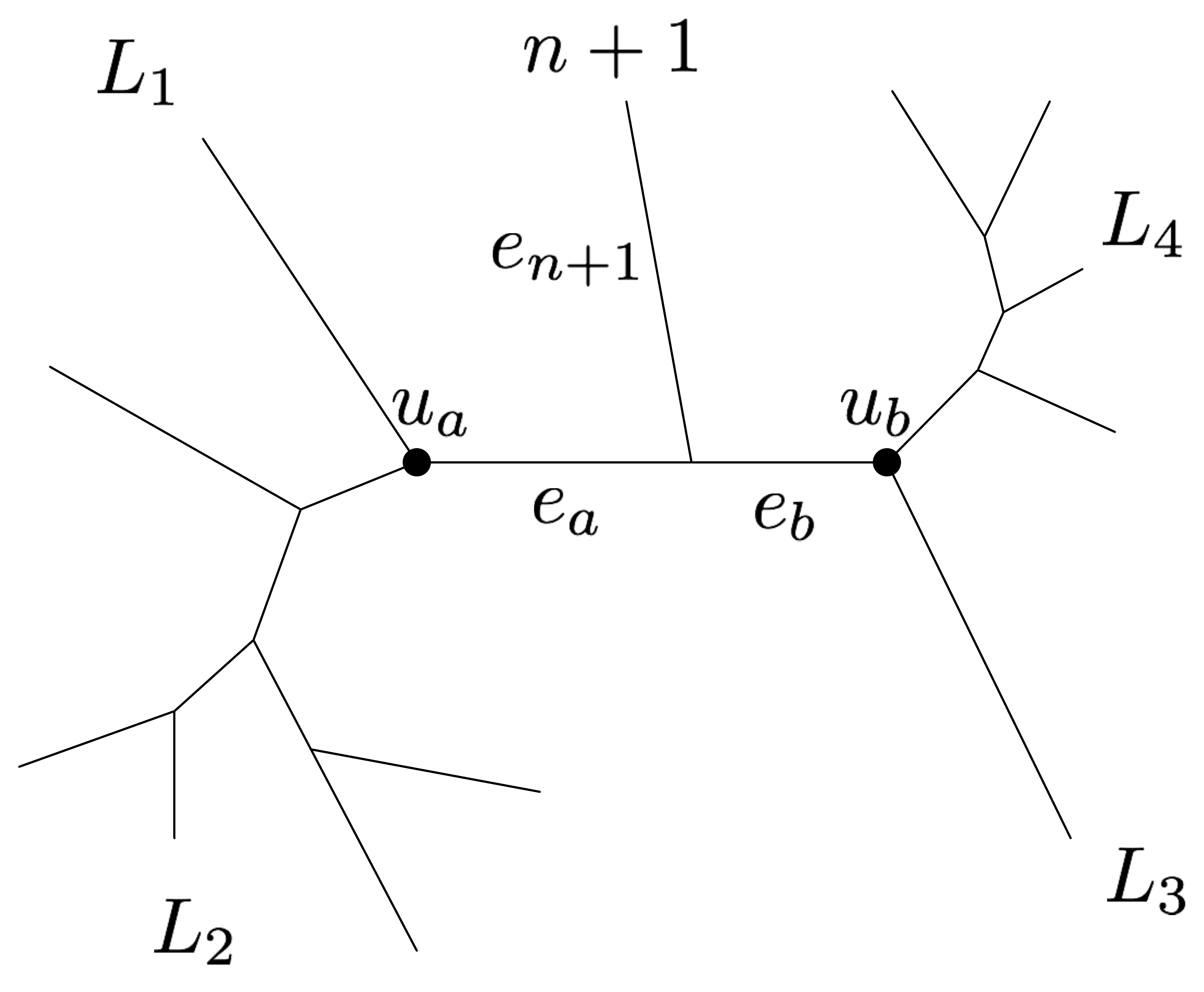}
\end{figure}

Delete the entry of $v_1$ and that of $v_2$ corresponding to the parameter $y_e$ to form 
$v_1', v_2' \in \mathbb{R}^{2n - 4}$. 
Define $w_1 = (v_1', w_1^a, w_1^b, w_1^{n+1})$ and 
 $w_2 = (v_2', w_2^a, w_2^b, w_2^{n+1})$
  where the entries of $w$ correspond to the edges of 
  $\mathcal{T}$ in the obvious way. 

Our goal will be to choose the six new vector entries so that
 $s$ wins $\alpha_{ij}^\mathcal{T} $ at $w$
  if and only if 
  $s$ wins $\alpha_{ij}^\mathcal{R}$ at $v$ 
  for $1 \leq i < j \leq n$.
   Moreover, we will want both 1 and 2 to win one of 
   $\{ \alpha^\mathcal{T}_{L_1,n + 1}, \alpha^\mathcal{T}_{L_2,n + 1} \}$
    and  $\{ \alpha^\mathcal{T}_{L_3,n + 1},\alpha^\mathcal{T}_{L_4,n + 1}\}$. 
    First, we will see why this will guarantee that $\dim_{\rr} \langle D^T_1(w) \rangle_\rr = \dim_{\rr} \langle D^T_2(w) \rangle_\rr  = 2(n+1) - 5$.

Form the matrix $A(\mathcal{T})$ with rows equal to \emph{all} the vectors $\alpha^\mathcal{T}_{ij}$. Let 

 $$
 \omega =
 \begin{pmatrix}
\omega(e_1) \\ 
\vdots \\ 
\omega(e_{2(n+1)-3}) \\
\end{pmatrix},
$$
be a vector of edge lengths for $\mathcal{T}$. 
Since $\alpha^\mathcal{T}_{ij} \cdot \omega$ 
gives us the distance between leaves $i$ and $j$ in $\mathcal{T}$, $A(\mathcal{T})\omega$ determines a metric on the leaves of $\mathcal{T}$. By the \emph{Tree-Metric theorem} (\cite{Pachter2004, Semple2003})  $\omega$ is the unique solution to $A(\mathcal{T})x = A(\mathcal{T})\omega$. Therefore, the rank of $A(\mathcal{T})$ is $2n - 3$.
Thus, if we can uniquely recover all of the edge lengths assigned to $\mathcal{T}$ from a matrix, the matrix has rank
at least $2(n+1) - 3$.

Let $\omega'$ be a vector of edge lengths for $\mathcal{R}$ where the lengths of edges shared between $\mathcal{R}$ and $\mathcal{T}$ are the same and $\omega'(e) = \omega(e_a) + \omega(e_b)$. 
Form the matrix $M^\mathcal{R}_s(v)$ 
with rows equal to the vectors in $D^\mathcal{R}_s(v)$.
By induction, this matrix has
rank equal to $2n -5$. Let $M^\mathcal{R}_s(v)' $ be the matrix $M^\mathcal{R}_s(v)$ augmented with two additional columns from $A(\mathcal{R})$ so that rank$(M^\mathcal{R}_s(v)') = 2n - 3$. 
Since this matrix is full rank, there is again a unique solution to 
$$
M^\mathcal{R}_s(v)'x = M^\mathcal{R}_s(v)'\omega'.
$$
This implies that we can uniquely determine the lengths of all $2n - 3$ edges in $\mathcal{R}$. 
As a corollary, we can recover the lengths of all edges in $\mathcal{T}$ that are also in $\mathcal{R}$
and $\omega(e_a) + \omega(e_b)$, the sum of the lengths of edges $e_a$ and $e_b$ in $\mathcal{T}$. 

Without loss of generality, suppose we have constructed $w = (w_1,w_2)$ so that 
$M^\mathcal{T}_1(v)$
 contains all of the columns from $M^\mathcal{R}_1(v)$ and
 columns corresponding to  
 $\alpha^\mathcal{T}_{L_1,n + 1} $ and 
 $ \alpha^\mathcal{T}_{L_3,n + 1}$. 
 Then let $M^\mathcal{T}_1(v)'$ be the matrix
 that contains all of the columns from $M^\mathcal{R}_1(v)'$ and columns corresponding to  
 $\alpha^\mathcal{T}_{L_1,n + 1} $ and 
 $ \alpha^\mathcal{T}_{L_3,n + 1}$.
 These columns enable us to recover the lengths of the paths
 from $L_1$ to $(n + 1)$ and from $L_3$ to $(n+1)$ in $\mathcal{T}$.
We will now show
 how this will enable us to determine
the lengths of the remaining edges, $e_{n+1}$, $e_a$, and $e_b$ uniquely. As explained, 
being able to determine all of the edge lengths of $\mathcal{T}$
from $M^\mathcal{T}_1(v)'$ shows that $M^\mathcal{T}_1(v)'$
has rank $2(n+1) - 3$.

Since we know the length of the path from $L_1$ to $(n+1)$ and the length of every edge between $L_1$ and $(n+1)$ except $e_{n+1}$ and $e_a$, we can determine $\omega(e_{n+1}) + \omega(e_a)$. 
Likewise, we know the length of the path from $L_3$ to $(n+1)$ and 
the length of
every edge between $L_3$ and $(n+1)$ except $e_{n+1}$ and $e_b$, so we can recover $\omega(e_{n+1}) + \omega(e_b)$. 
Combined with our knowledge of $\omega(e_a) + \omega(e_b)$
 we can determine the lengths of $e_{n+1}$, $e_a$, and $e_b$. Uniqueness implies that the augmented matrix 
$M^\mathcal{T}_1(w)'$ has rank $2(n+1) - 3$ 
and so $M^\mathcal{T}_1(w)$ has rank $2(n+1) - 5$ as desired.
If we have also chosen 
$w = (w_1,w_2)$ so that 
$M^\mathcal{T}_2(v)$
 contains all of the columns from $M^\mathcal{R}_2(v)$ and
 columns corresponding to  
 $\alpha^\mathcal{T}_{L_2,n + 1} $ and 
 $ \alpha^\mathcal{T}_{L_4,n + 1}$, then 
the same is true for $M^\mathcal{T}_2(w)$, and the theorem is complete.

It remains to show that we can actually choose the six new 
vector entries $w_1^a, w_1^b, w_1^{n+1}, \\w_2^a, w_2^b,$ and $w_2^{n+1}$ in the manner specified. First, note that every edge in $\mathcal{T}$ along the path from $u_a$ to $L_1$ or $L_2$ and $u_b$ to $L_3$ or $L_4$ is contained in $\mathcal{R}$. Therefore, we let $a^s_i$ be the $v_s$-weight of the path from $u_a$ to $L_i$ with $i = 1,2$ and we have:

\begin{align*}
w_1 \cdot \alpha^\mathcal{T}_{L_1,n+1} &= a^1_1 + w_1^a + w_1^{n+1},\\
w_2 \cdot \alpha^\mathcal{T}_{L_1,n+1} &= a^2_1 + w_2^a + w_2^{n+1}, \\
w_1 \cdot \alpha^\mathcal{T}_{L_2,n+1} &= a^1_2 + w_1^a + w_1^{n+1}, \\
w_2 \cdot \alpha^\mathcal{T}_{L_2,n+1} &= a^2_2 + w_2^a + w_2^{n+1}.\\
\end{align*}

  Recall that our goal is for  both 1 and 2 to win one of 
   $\{ \alpha^\mathcal{T}_{L_1,n + 1}, \alpha^\mathcal{T}_{L_2,n + 1} \}$
    and \\ $\{ \alpha^\mathcal{T}_{L_3,n + 1},\alpha^\mathcal{T}_{L_4,n + 1}\}$. 
    Rearranging, we would like to have
\begin{align*}
 a^1_1 + w_1^a + w_1^{n+1} &< a^2_1 + w_2^a + w_2^{n+1}, \\
 a^1_2 + w_1^a + w_1^{n+1} &> a^2_2 + w_2^a + w_2^{n+1}\\
 \Rightarrow 
 (w_1^a + w_1^{n+1}) - (w_2^a + w_2^{n+1}) &< a^2_1 - a^1_1\\
 (w_1^a + w_1^{n+1}) - (w_2^a + w_2^{n+1}) &> a^2_2 - a^1_2\\.
\end{align*}

If we let $ (w_1^a + w_1^{n+1}) =  (a^2_1 - a^1_1)/2$ and $(w_2^a + w_2^{n+1}) = -(a^2_2 - a^1_2)/2$ then  $(w_1^a + w_1^{n+1}) - (w_2^a + w_2^{n+1})$ is just the average of $(a^2_1 - a^1_1)$ and $(a^2_2 - a^1_2)$. For $w$ chosen sufficiently generic, the inequalities above may both be switched, but regardless, we will have sent the vectors $\{ \alpha^T_{L_1,n + 1}, \alpha^T_{L_2,n + 1} \}$ into different matrices. By symmetry, we let $b^s_i$ be the $v_s$-weight of the path from $u_b$ to $L_i$ for $i = 3,4$. Then we will be done if the following system has a solution:

\begin{align*}
w_1^a + w_1^{n+1} &=  (a^2_1 - a^1_1)/2 \\
w_2^a + w_2^{n+1} &= -(a^2_2 - a^1_2)/2 \\
w_1^b + w_1^{n+1} &=  (b^2_1 - b^1_1)/2 \\
w_2^b + w_2^{n+1} &= -(b^2_2 - b^1_2)/2 \\
w_1^a + w_1^b &= v_1^e \\
w_2^a + w_2^b &= v_2^e. \\
\end{align*}

The last two equations are necessary so that $s$ wins $\alpha_{ij}^\mathcal{T}$ at $w$ if and only if $s$ wins $\alpha_{ij}^\mathcal{R} $ at $v$. The resulting matrix is full rank. \end{proof}

Finally, we have all of the pieces necessary to complete our proof.
\begin{proof}[Proof of Theorem \ref{gtisecants}]
Corollary \ref{cherrycorollary} shows that if $\mathcal{T}$ has
five or more cherries then
 $in_{\omega}(I_{2,n}^{\{2\}})  \not = J_\mathcal{T}^{\{2\}}.$
 Let $\mathcal{T}$ be a binary phylogenetic $[n]$-tree  with fewer than $5$ cherries. By Lemma \ref{jtdraisma}, $\dim(J_\mathcal{T}^{\{2\}}) \geq  4n - 10$, and since 
 $in_{\omega}(I_{2,n}^{\{2\}}) \subseteq 
  J_\mathcal{T}^{\{2\}}$, and 
 $\dim(in_{\omega}(I_{2,n}^{\{2\}})) =4n-10$, 
$\dim(in_{\omega}(I_{2,n}^{\{2\}})  ) = \dim(J_\mathcal{T}^{\{2\}}).$ 
By Theorem \ref{primesecants}, $in_{\omega}(I_{2,n}^{\{2\}})$ is prime and of the same dimension as $J_\mathcal{T}^{\{2\}}$, which implies
$in_{\omega}(I_{2,n}^{\{2\}})   = J_\mathcal{T}^{\{2\}}.$
\end{proof}

\section{Beyond the Second Secant}
\label{Beyond the Second Secant}

Based on the proof of Theorem \ref{gtisecants} and the result of  Lemma \ref{primesecants}, we have the following corollary which is a modification of the statement of Conjecture \ref{sethsconjecture}.

\begin{cor}
\label{primesecantscor} 
Let $\mathcal{T}$ be a binary phylogenetic $[n]$-tree , 
$\omega \in \mathbb{R}^n$ a weight vector, 
and $\tau \in \{\pm1\}^{n \choose 2}$ a sign vector such that 
$J_\mathcal{T} = \tau \cdot in_\omega(I_{2,n})$. 
Then 
$\tau \cdot in_\omega(I_{2,n}^{\{r\}}) = J_\mathcal{T}^{\{r\}}$ 
if and only if
 $\dim(J_\mathcal{T}^{\{r\}}) = 2rn - 2r^2 - r $.
\end{cor}

We have already seen that Conjecture \ref{sethsconjecture} is not true for trees with more than $2r$ cherries. However, as $r$ increases, the number of cherries is not the only obstruction. The presence of other tree structures factors into a bound on the possible dimension of $J_\mathcal{T}^{\{r\}}$. 

Removing an edge from a binary phylogenetic $[n]$-tree creates two connected components each of which is a rooted binary phylogenetic $K$-tree for some $K \subset [n]$. If one of these rooted trees is a $k$-leaf rooted caterpillar then we call this rooted subtree a \emph{k-cluster} of $\mathcal{T}$. 
Cherries, then, may alternatively be referred to as 2-clusters. 
We let $c_k$ be the number of $k$-clusters in a tree. If leaves $i$ and $j$ are contained in an $s$-cluster, then we let $k$ be the smallest such $s$ and call the variable $p_{ij}$ a
 \emph{k-cluster variable for $\mathcal{T}$}.

\begin{ex}
Let $\mathcal{T}$ be the tree in Figure \ref{cluster13}.
Then  $\mathcal{T}$ has three 3-clusters on the leaves $\{1,2,3\}$, $\{4,5,6\}$, and $\{11,12,13\}$. The set of 2-cluster variables is $$\{p_{1,2}, p_{4,5},p_{7,8},p_{9,10},p_{11,12}\}$$ and the set of 3-cluster variables is
$$\{p_{1,3}, p_{2,3},p_{4,6},p_{5,6},p_{11,13}, p_{12,13}\}.$$
\end{ex}
Notice that the way clusters are nested, the number of $k$-cluster variables in a tree will be $(k-1)c_k$.

\begin{figure}[h!]
  \caption{A 13-leaf tree with five 2-clusters and three 3-clusters.}
  \label{cluster13}
  \centering
    \includegraphics[width=0.4\textwidth]{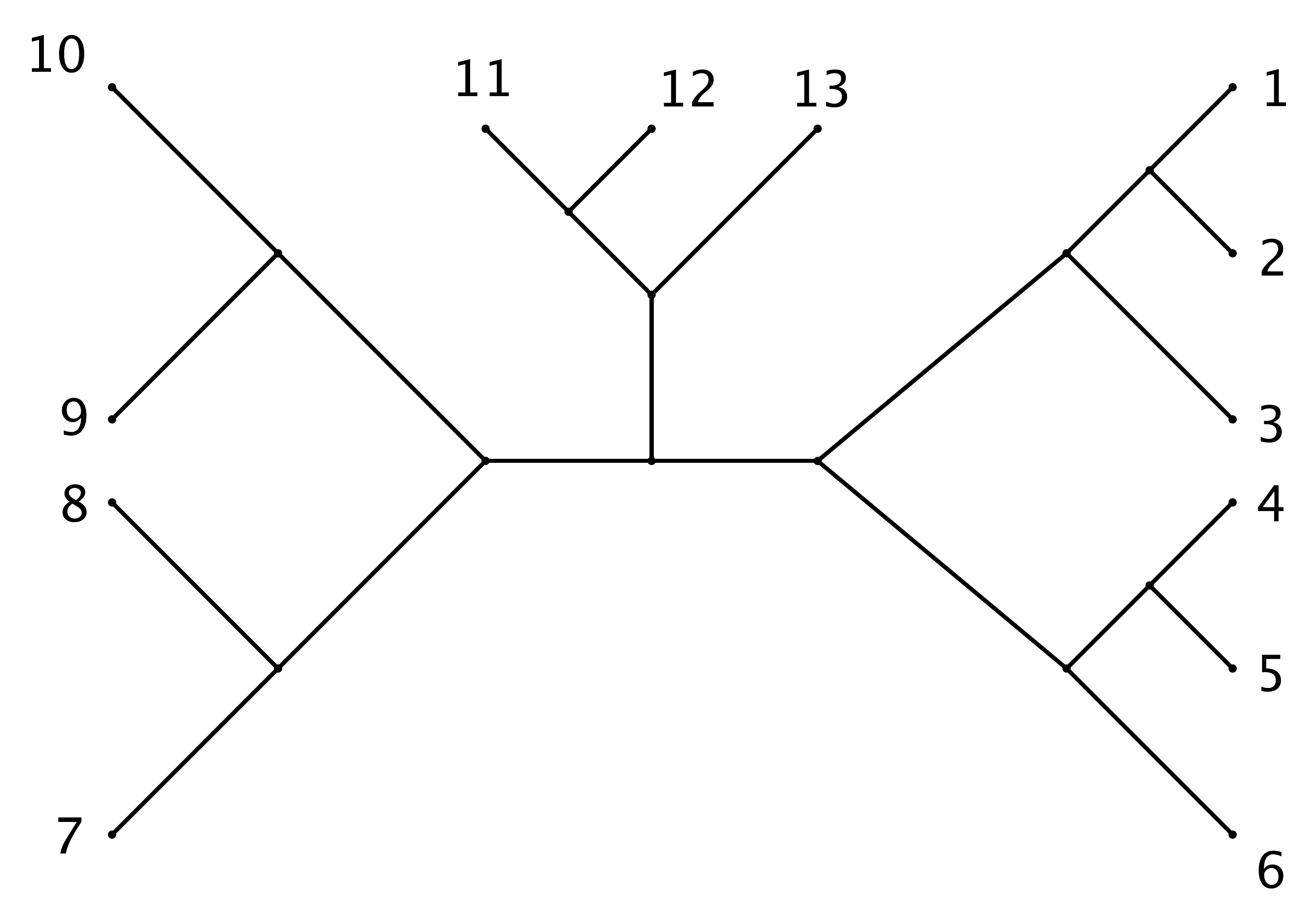}
\end{figure}

\begin{lemma}
\label{clusterdimbound}
 Let $c_k$ be the number of $k$-clusters in $\mathcal{T}$, then
  $$\dim(J_\mathcal{T}^{\{r\}}) \leq 2rn - 3r - \displaystyle\sum_{k=2}^{r} (r - k + 1)c_{k}.$$
\end{lemma}

\begin{proof} 
Let $\omega$ be a weight vector such that $J_\mathcal{T} = in_\omega(I_{2,n})$. Let $J_\mathcal{T}^E$ be the ideal constructed  by eliminating all $k$-cluster variables from $J_\mathcal{T}$ for $1 \leq k \leq r-1$,  and embedding this ideal in $Z^n$.
Define  $V(J_\mathcal{T}^E) = W$ and
note that $V(J_\mathcal{T}) \subseteq W$ and 
$\dim(J_\mathcal{T}^{\{r\}}) \leq \dim(\mathcal{I}(W)^{\{r\}})$. 
There are no restrictions on the $\sum_{k=2}^{r} (k-1)c_{i}$ 
eliminated variables in $W$, so we may write 
 $W = W' \times \mathbb{C}^{\sum_{k=2}^r (k-1)c_{i}}$. 
Since $\mathbb{C}^{\sum_{k=2}^r (k-1)c_{k}}$ is a linear space, 
$W^{\{r\}} = 
W'^{\{r\}}\times{\mathbb{C}^{\sum_{k=2}^r (k-1)c_{k}}}$. 
We also observe that $W'$ is a cone since it is a coordinate
projection of a cone. 
Thus, $\dim(W^{\{r\}}) \leq r\dim(W'^{\{r\}}) + \sum_{k=2}^r (k-1)c_{k}$. 

Now we seek a bound for $\dim(W'^{\{r\}})$. Choose a specific $k$-cluster in $\mathcal{T}$, define a grading with every $k$-cluster variable in that $k$-cluster having weight one and every other variable having weight zero. Observe that each binomial generator of $J_\mathcal{T}$ is homogeneous with respect to this grading. 
 Thus, the equations in any reduced Gr\"obner basis for $J_\mathcal{T}$ with respect to any monomial order contain at least two distinct $k$-cluster variables from the designated
 $k$-cluster if they contain any at all. 
 If not, by homogeneity, there exists an equation in the reduced
Gr\"obner basis in which $p_{ij}$, a $k$-cluster variable from the designated
 $k$-cluster, can be factored. Since $J_\mathcal{T}$ is prime, that implies that $p_{ij}$ is zero, which is evidently not true from the parameterization. Therefore, choosing an elimination order and eliminating any $(k-2)$ of the designated $k$-cluster variables from $J_\mathcal{T}$ eliminates all of the $(k-1)$ designated $k$-cluster variables. Thus, projecting away all of the $k$-cluster variables from a given $k$-cluster in $V(J_\mathcal{T})$ yields a variety of at least one dimension less. Applying the same argument to each $k$-cluster implies $\dim(W') \leq 2n - 3 - \sum_{k=2}^{r} c_k$, and the result follows.
\end{proof}

In the case where $r=2$, this is just a restatement of Lemma \ref{dimbound}. When $r=3$, we have $\dim(I_{2,n}^{\{3\}}) = 6n - 21$, so this tells us that it is impossible for $J_\mathcal{T}^{\{3\}} = I_{2,n}^{\{3\}}$ when $2c_2 + c_3  > 12$. 


\begin{ex}
\label{13leafex}
Let $\mathcal{T}$ be the $13$-leaf tree pictured in Figure \ref{cluster13}.
Then $c_2 = 5$, $c_3 = 3$, and $2c_2 + c_3  = 13$. Lemma \ref{clusterdimbound}  tells us that $\dim(J_\mathcal{T}^{\{3\}}) \leq 66 < 67 = \dim(I_{2,n}^{\{3\}})$ so that $J_\mathcal{T}^{\{3\}}\not=I_{2,n}^{\{3\}}$. Evaluating  the Jacobian matrix at a
random point we find that $\dim(J_\mathcal{T}^{\{3\}}) = 66$.
\end{ex}

\bigskip
\bigskip

Finally, one might wonder if we can modify Conjecture \ref{sethsconjecture} as follows.

\begin{conj}
\label{highersecconj} 
Let $\mathcal{T}$ be a binary phylogenetic $[n]$-tree, 
$\omega \in \mathbb{R}^n$ a weight vector, 
and $\tau \in \{\pm1\}^{n \choose 2}$ a sign vector such that 
$J_\mathcal{T} = \tau \cdot in_\omega(I_{2,n})$. 
Then 
$\tau \cdot in_\omega(I_{2,n}^{\{r\}}) = J_\mathcal{T}^{\{r\}}$ 
if and only if
  $$\displaystyle\sum_{k=2}^{r} (r - k + 1)c_{k} < 2r^2 - 2r.$$
\end{conj}

We have investigated $\dim(J_\mathcal{T}^{\{r\}})$ for $r=3$ and $r=4$ and several trees up to 18 leaves. By evaluating the Jacobian matrix at random points, we have found in each case that the conjecture holds. It may be possible to prove Conjecture \ref{clusterdimbound} utilizing induction as we did in Lemma \ref{jtdraisma}, however, there are many more base cases
to handle and it is unclear how to generalize the induction step.

\section*{Acknowledgements}

Colby Long was partially supported by the US National Science Foundation (DMS 0954865). We would like to thank Seth Sullivant for suggesting this problem and for many helpful comments during the completion of this work.
\bibliography{references}
\bibliographystyle{plain}

\end{document}